\documentclass[a4paper,12pt]{article}
\usepackage[dvipdfmx]{graphicx}
\usepackage{amsmath,amssymb,amsfonts,amsthm}
\newtheorem{definition}{Definition}[section]
\newtheorem{theorem}[definition]{Theorem}
\newtheorem{lemma}[definition]{Lemma}

\newtheorem{example}[definition]{Example}
\newtheorem{proposition}[definition]{Proposition}

\title{Some families of operator norm inequalities}
\author{Imam Nugraha Albania and Masaru Nagisa}

\begin{document}
\maketitle

\begin{abstract}
We consider the function $ f_{\alpha, \beta}(t)=t^{\gamma(\alpha,\beta)}\prod_{i=1}^n \frac{b_i(t^{a_i}-1)}{a_i(t^{b_i}-1)}$ on the interval $(0,\infty)$,
where $\alpha=(a_1,a_2,\ldots,a_n), \beta=(b_1,b_2,\ldots,b_n)\in \mathbb{R}^n$ and $\gamma(\alpha,\beta) = (1-\sum_{i=1}^n(a_i-b_i))/2$.
In \cite{hiaikosaki}, Hiai and Kosaki define the relation $\preceq $ using positive definiteness  for functions $f$ and $g$ with some suitable conditions and
they have proved this relation implies the operator norm inequality associated  with functions $f$ and $g$.
In this paper, we give some conditions for $\alpha', \beta'\in \mathbb{R}^m$ to hold the relation $f_{\alpha,\beta}(t) \preceq f_{\alpha', \beta'}(t)$.
\end{abstract}

\section{Introduction}
When $f: (0,\infty)\longrightarrow (0,\infty)$ is continuous and satisfies $f(1)=1$,
we denote $f \in C(0,\infty)^+_1$.
We call $f\in C(0,\infty)^+_1$ symmetric if it holds $f(t)=tf(1/t)$.
For $f, g\in C(0,\infty)^+_1$, we define  $f \preceq g$ if the function
\[   \mathbb{R} \ni x \mapsto \frac{f(e^x)}{g(e^x)}  \]
is positive definite, 
where a function $\varphi : \mathbb{R} \longrightarrow \mathbb{C}$ is positive definite means that, for any positive integer $n$ and real numbers $x_1, x_2,\ldots, x_n$,
the $n\times n$ matrix  $[ \varphi (x_i - x_j)]_{i,j=1}^n$ is positive definite, i.e.,
\[   \sum_{i,j=1}^n \alpha_i \overline{\alpha_j} \varphi(x_i-x_j) \ge 0 \]
for any $\alpha_1, \alpha_2, \ldots, \alpha_n \in \mathbb{C}$.  
For $f\in C(0,\infty)^+_1$, we define a continuous map 
$M_f: (0,\infty)\times (0,\infty) \longrightarrow (0,\infty)$ as follows:
\[    M_f(s, t) = t f(\frac{s}{t}).  \]
Then it holds that $M_f(1,1)=1$, $M_f(\alpha s, \alpha t) = \alpha M_f(s , t)$  $(\alpha >0)$ and
\[   M_f(s,t) =M_f(t,s)   \]
if $f$ is symmetric.

We define the inner product $\langle \cdot, \cdot \rangle$ on $\mathbb{M}_N(\mathbb{C})$ by $\langle X, Y \rangle = {\rm Tr} (Y^*X)$
for $X, Y\in \mathbb{M}_N(\mathbb{C})$.
When $A\in\mathbb{M}_N(\mathbb{C})$, we can define bounded linear operator $L_A$ and $R_A$ on the Hilbert space 
$(\mathbb{M}_N(\mathbb{C}), \langle \cdot,\cdot \rangle)$ as follows:
\[  L_A(X) =AX, \quad R_A(X) = XA \quad \text{ for } X\in \mathbb{M}_N(\mathbb{C}). \]
If both $H$ and  $K$ are positive, invertible matrix in $\mathbb{M}_N(\mathbb{C})$ (in short, $H, K>0$), 
then $L_H$ and $R_K$ are also positive, invertible operators on $(\mathbb{M}_N(\mathbb{C}), \langle \cdot,\cdot \rangle)$
and satisfy the relation $L_HR_K = R_KL_H$.
Using continuous function calculus of operators, we can consider the operator 
$M_f(L_H, R_K) (= R_K f(L_HR_K^{-1}))$ on $(\mathbb{M}_N(\mathbb{C}), \langle \cdot,\cdot \rangle)$.

In \cite{hiaikosaki},  F. Hiai and H. Kosaki has given the following equivalent conditions for 
$f,g \in C(0,\infty)_1^+$ satisfying the symmetric condition:
\begin{enumerate}
  \item[(1)]  there exists a symmetric probability measure $\nu$ on $\mathbb{R}$ such that
\[  M_f(L_H, R_K)X = \int_{-\infty}^\infty H^{is}( M_g(L_H, R_K)X)K^{-is}d\nu(s)  \]
for all $H,K,X\in \mathbb{M}_N(\mathbb{C})$ with $H,K>0$.
  \item[(2)] $||| M_f(L_H, R_K)X ||| \le ||| M_g(L_H, R_K)X ||| $ for all $H,K,X\in \mathbb{M}_N(\mathbb{C})$ with $H,K>0$
and  any unitarily invariant norm $|||\cdot |||$, which means $||| UX ||| = |||X||| = |||XU|||$ 
for any unitary $U\in \mathbb{M}_N(\mathbb{C})$ and any matrix $X\in \mathbb{M}_N(\mathbb{C})$.
  \item[(3)] $\| M_f(L_H, R_H)X  \| \le  \| M_g(L_H, R_H)X  \| $ for all $H,X\in \mathbb{M}_N(\mathbb{C})$ with $H>0$ and the usual operator norm
$\|\cdot\|$ on $\mathbb{M}_N(\mathbb{C})$.
  \item[(4)] $f \preceq g$.
\end{enumerate}
They also proved that, for a family of symmetric functions $f_a(t) = \frac{a-1}{a}\frac{t^a-1}{t^{a-1}-1} \in C(0,\infty)^+_1$ $(a\in \mathbb{R})$, 
\[  -\infty \le a < b \le \infty \; \Rightarrow \; f_a\preceq f_b  .  \]
As an example, $f_{1/2} \preceq f_2$  implies 
\[   ||| M_{f_{1/2}}(L_H, R_K)X ||| \le  ||| M_{f_2}(L_H, R_K)X ||| .  \]
So we can get the arithmetic-geometric mean inequality
\[   |||H^{1/2}XK^{1/2} ||| \le \frac{1}{2}||| HX+XK |||,  \]
because $ M_{f_{1/2}}(s,t) =s^{1/2}t^{1/2}$ and  $M_{f_2}(s,t) =(s+t)/2$.
This is known as McIntosh's inequality \cite{mcintosh}.

In this paper, we consider the following function:
\[   f_{\alpha, \beta}(t) = t^{\gamma(\alpha, \beta)} 
                   \prod_{i=1}^n \frac{b_i (t^{a_i}-1)}{a_i (t^{b_i}-1)}  \]
for $\alpha=(a_1,a_2,\ldots, a_n)$, $\beta=(b_1,b_2,\ldots, b_n)\in \mathbb{R}^n$
and $\gamma(\alpha,\beta)=(1-\sum_{i=1}^n (a_i -b_i) )/2$.
Under some condition, the second-named author investigated their operator monotonicity in \cite{nagisawada}.
The function $f_{\alpha,\beta} \in C(0,\infty)^+_1$ is an extension of  functions $\{f_a:a\in \mathbb{R} \}$ in some sense and satisfies the symmetric condition.
We also set
\[  M_{\alpha, \beta} (s, t) = t f_{\alpha,\beta}(\frac{s}{t} ) .  \]
For $\alpha=(a_1,a_2,\ldots, a_n)$, $\beta=(b_1,b_2,\ldots, b_n)\in \mathbb{R}^n$, we define the relation $|\alpha| \preceq_w |\beta|$ as follows:
\begin{gather*}
  |a_{\sigma (1)}| \ge |a_{\sigma (2)}| \ge \cdots \ge |a_{\sigma (n)}|, \; 
  |b_{\tau (1)}| \ge |b_{\tau (2)}| \ge \cdots \ge |b_{\tau (n)}|  \\
  \text{and } \; \sum_{i=1}^k |a_{\sigma (i)}| \le \sum_{i=1}^k |b_{\tau (i)}| \qquad  (k=1,2,\cdots, n)
\end{gather*}
for some permutations $\sigma, \tau$ on $\{1,2,\ldots, n\}$,
where we denote $(|a_1|,|a_2|,\ldots, |a_n|)$ by $|\alpha|$.
In this case we call that $|\beta|$ weakly submajorises $|\alpha|$.

Our main result is as follows:

\vspace{5mm}

\begin{theorem}
Let
\[  f_{\alpha,\beta}(t) = t^{\gamma(\alpha, \beta)} 
                   \prod_{i=1}^n \frac{b_i (t^{a_i}-1)}{a_i (t^{b_i}-1)} , \quad
    f_{\alpha', \beta'} = t^{\gamma(\alpha', \beta')}
                   \prod_{i=1}^m \frac{d_i (t^{c_i}-1)}{c_i (t^{d_i}-1)} , \]
where $\gamma(\alpha,\beta) = (1-\sum_{i=1}^n (a_i-b_i) )/2$,
$\gamma(\alpha',\beta') = (1-\sum_{i=1}^m (c_i-d_i) )/2$.
If  $|(b_1,\ldots, b_n,c_1,\ldots c_m)|$ weakly submajorises $|(a_1,\ldots,a_n,d_1,\ldots, d_m)|$,
then we have  $f_{\alpha,\beta} \preceq f_{\alpha',\beta'}$, that is,
\[   ||| M_{\alpha, \beta}(L_H, R_K)X ||| \le ||| M_{\alpha', \beta'}(L_H, R_K)X |||   \]
for any  $H,K \in \mathbb{M}_N(\mathbb{C})$ with $H,K>0$
and any matrix $X\in \mathbb{M}_N(\mathbb{C})$.
\end{theorem}

We can get an operator norm inequality for a pair of sequences of positive numbers
if one sequence is weakly submajorise the other one.
In a special case, we can completely determine the condition to get the related operator norm inequality.

\vspace{5mm}

\begin{theorem}
Let $a,b,c,d  \ge 0$ and set
\[  f_{a,b}(t) = t^{(1-a+b)/2}\frac{b(t^a-1)}{a(t^b-1)}, \;
    f_{c,d}(t) = t^{(1-c+d)/2}\frac{d(t^c-1)}{c(t^d-1)}.  \]
\begin{enumerate}
\item[(1)] When $a\ge b$, $f_{a,b}\preceq f_{c,d}$ is equivalent to
\[  (c,d) \in \{ (x,y) : x\ge a, \; 0\le y \le x-a+b  \} .  \]
\item[(2)] When $a<b$, $f_{a,b}\preceq f_{c,d}$ is equivalent to
\[  (c,d) \in \{ (x,y) : 0\le x \le y \le x-a+b, \; y\le b  \} \cup \{ (x,y) : 0 \le y \le x\} .  \]
\end{enumerate}
\end{theorem}

We remark that this statement has been proved in \cite{kanemitsu} based on the facts given by \cite{hiaikosaki} and \cite{hkpr}.


\section{Positive Definite Functions and Infinitely Divisible Functions}

We call a function $\varphi : \mathbb{R}\longrightarrow \mathbb{C}$ positive definite
if, for any positive integer $n$ and any real numbers 
$x_1, x_2, \ldots, x_n\in \mathbb{R}$, the matrix
\[  \begin{pmatrix}  
     \varphi(0) & \varphi(x_1-x_2) & \cdots & \varphi(x_1-x_n) \\
     \varphi(x_2-x_1) & \varphi(0) & \cdots & \varphi(x_2-x_n) \\
     \vdots & \vdots & \ddots & \vdots \\
     \varphi(x_n-x_1) & \varphi(x_n-x_2) & \cdots & \varphi(0)
   \end{pmatrix}  \]
is positive, that is,
\[  \sum_{i,j=1}^n \alpha_i\overline{\alpha_j} \varphi (x_i-x_j) \ge 0  \quad \text{ for } \alpha_1, \alpha_2,\ldots, \alpha_n\in \mathbb{C} . \]
By definition, it easily follows that the function $x\mapsto e^{iax}$ is positive definite
for any $a\in \mathbb{R}$.
This implies the Fourier transform $\hat{\mu}(x) = \int_{-\infty}^\infty e^{ixt}d\mu(t)$ 
of  a finite positive measure $\mu$ on 
$\mathbb{R}$ is positive definite.
As Bochner's theorem \cite{bhatia2}, 
it is known that $\varphi$ is positive definite and continuous at $0$ if and only if
there exists a finite, positive measure $\mu$ on $\mathbb{R}$ satisfying
\[  \varphi(x) = \int_{-\infty}^\infty e^{ixt}d\mu(t)  .   \]

\vspace{5mm}


\begin{lemma}
Let $\varphi, \varphi_1, \varphi_2, \ldots$ be positive definite and $\psi$ be the point-wise limit of
the sequence $\{\varphi_n\}_{n=1}^\infty$.
\begin{enumerate}
  \item[(1)] For a positive real number $a,b$,  $a\varphi_1+b\varphi_2$ is positive definite. .
  \item[(2)] $\psi$ is positive definite.
  \item[(3)] The product $\varphi_1 \varphi_2$ of $\varphi_1$ and $\varphi_2$ is positive definite.  
  \item[(4)] $e^\varphi$ is positive definite.
\end{enumerate}
\end{lemma}
\begin{proof}
(1) and (2) easily follow by definition.

(3) When $A=(a_{ij})$ and $B=(b_{ij})$ are positive matrices in $\mathbb{M}_n(\mathbb{C})$,
the Schur product $A\circ B = (a_{ij}b_{ij})$ of $A$ and $B$ is also positive.
Applying this fact for a matrix 
$(\varphi_1(x_i - x_j) \varphi_2(x_i -x_j) )_{i,j=1}^n$
$(x_1, \ldots, x_n \in \mathbb{R})$, 
we can see $\varphi_1\varphi_2$ is positive definite.

(4) Since $e^{\varphi}(x) =e^{\varphi(x)}=\sum_{k=0}^\infty \varphi(x)^k/k!$, a matrix
\[  (e^{\varphi}(x_i-x_j) )_{i,j=1}^n =
   \lim_{m\to \infty} \sum_{k=0}^m \frac{1}{k!}( (\varphi(x_i-x_j)^k )_{i,j=1}^n \]
is positive definite by (1), (2), and (3).
So $e^\varphi$ is positive definite. 
\end{proof}

\vspace{5mm}

A positive definite function $\varphi$ is called infinitely divisible if $\varphi^r$ is
positive definite for any $r>0$.
When $\varphi$ is the Fourier transform of a probability measure $\mu$ on $\mathbb{R}$,
i.e., 
\[  \varphi(x) = \int_{-\infty}^\infty e^{ixt}d\mu(t) , \]
we call $\varphi$ the characteristic function of $\mu$.
It is known as L\'evi-Khintchine theorem that $\varphi$ is an infinitely divisible characteristic function if and only if
it can be written as
\[  \log \varphi(x) = i\gamma x 
       + \int_{-\infty}^\infty (e^{ixt}-1-\frac{ixt}{1+t^2})\frac{1+t^2}{t^2}d\nu(t)   \]
with a finite positive measure $\nu$ and $\gamma\in \mathbb{R}$.
It is also known as Kolmogorov's theorem that $\varphi$ is the characteristic function of an 
infinitely divisible probability measure $\mu$ with finite second moment if and only if
\[  \log \varphi (x) = i\gamma x + \int_{-\infty}^\infty (\frac{e^{itx}-1-itx}{t^2})d\nu (t)  \]
with a finite measure $\nu$ and $\gamma \in \mathbb{R}$ (\cite{kolmogorov}, \cite{kosaki}).

\vspace{5mm}


\begin{lemma}
Let $a,b$ be positive numbers and set
\[   f(x) = \frac{b\sinh ax}{a\sinh bx}  .  \]
Then the following are equivalent:
\begin{enumerate}
  \item[(1)]  $a\le b$.
  \item[(2)]  $f$ is positive definite. 
  \item[(3)]  $f$ is infinitely divisible.
\end{enumerate}
\end{lemma}
\begin{proof}
$(3)\Rightarrow(2)$ It is clear by definition.

$(2)\Rightarrow (1)$ We assume that $a>b$.
Since $f(0)=1$, $f(-x) = f(x)$, and $\lim_{x\to \infty}f(x)=\infty$,
\[  \begin{pmatrix} f(0) & f(-x)  \\ f(x) & f(0) \end{pmatrix}  \]
is not positive for a sufficiently large $x$.
So $f$ is not positive definite.
This means the positive definiteness of $f$ implies $a \le b$.

$(1)\Rightarrow(3)$
The function $f(x)$ can be written as
\[  \log f(x) = \int_{-\infty}^\infty (e^{ixt}-1-ixt)
                   \frac{\sinh ((1/a-1/b)\pi t/2)}{2t\sinh(\pi t/2a)\sinh(\pi t/2b)}dt \]
(\cite{kosaki}:Corollary 3).
So we have $f$ is infinitely divisible when $a\le b$.
\end{proof}

\vspace{5mm}

Using above integral expression of the function $\frac{b\sinh ax}{a\sinh bx}$,
H. Kosaki(\cite{kosaki}: Theorem 5) proved
\[  \frac{(\alpha - 1)\beta}{\alpha(\beta-1)}
     \frac{\sinh(\alpha x)\sinh((\beta-1)x)}{\sinh((\alpha-1)x)\sinh(\beta x)}  \] 
is infinitely divisible if $\beta\ge \alpha$.

\vspace{5mm}


\begin{lemma}
Let $a,b,c,d$ be positive numbers with $d> \max\{a,b,c\}$ and $a+c=b+d$.
Then we have
\[   f(x) = \frac{\sinh ax \sinh cx}{\sinh bx \sinh dx}  \]
is infinitely divisible.
\end{lemma}
\begin{proof} 
By the assumption, $d>\max \{a,c\} \ge \min\{a,c\} >b$.
Since $\frac{a}{a-b} -1 = \frac{b}{a-b}$, $\frac{d}{a-b}-1=\frac{c}{a-b}$, 
and $\frac{a}{a-b}<\frac{d}{a-b}$, we have the function
\[  x \mapsto \frac{\sinh \frac{a}{a-b}x \sinh\frac{c}{a-b}x}{\sinh \frac{b}{a-b}x \sinh\frac{d}{a-b}x} \]
is infinitely divisible.
So the function
\[  t \mapsto x=(a-b)t \mapsto  
     \frac{\sinh \frac{a}{a-b}x \sinh\frac{c}{a-b}x}{\sinh \frac{b}{a-b}x \sinh\frac{d}{a-b}x} 
     = \frac{\sinh at \sinh ct}{\sinh bt \sinh dt}  \]
is also infinitely divisible.
\end{proof}

\vspace{5mm}


\begin{lemma}
Let $a_i\ge a_i'>0$ and $0<b_i \le b_i'$  $(i=1,2,\ldots, n)$. 
\begin{align*}
  (1) & \;  \prod_{i=1}^n \frac{\sinh a_ix}{\sinh b_ix} \text{ is infinitely divisible }
    \Rightarrow \prod_{i=1}^n \frac{\sinh a_i'x}{\sinh b_i'x} \text{ is infinitely divisible}.  \\
  (2) & \;  \prod_{i=1}^n \frac{\sinh a_i'x}{\sinh b_i'x} \text{ is not positive definite }
    \Rightarrow \prod_{i=1}^n \frac{\sinh a_ix}{\sinh b_ix} \text{ is not positive definite}. 
\end{align*}
\end{lemma}
\begin{proof}
(1)  For any $r>0$, we have
\[   (\prod_{i=1}^n\frac{\sinh a_i'x}{\sinh b_i'x} )^r =(\prod_{i=1}^n \frac{\sinh a_ix}{\sinh b_ix})^r
      \prod_{i=1}^n\bigl( (\frac{\sinh a_i'x}{\sinh a_ix})^r (\frac{\sinh b_i x}{\sinh b_i' x})^r\bigr) .  \]
By the assumption, Lemma 2.1 and Lemma 2.2, we can get the infinite divisibility of
\[   \prod_{i=1}^n\frac{\sinh a_i'x}{\sinh b_i'x} .  \]

(2) It suffices to show that $\prod_{i=1}^n\frac{\sinh a_i'x}{\sinh b_i'x}$ is positive definite
when $\prod_{i=1}^n\frac{\sinh a_ix}{\sinh b_ix}$ is positive definite.
By the identity
\[   \prod_{i=1}^n\frac{\sinh a_i'x}{\sinh b_i'x}  =\prod_{i=1}^n \frac{\sinh a_ix}{\sinh b_ix}
      \prod_{i=1}^n\bigl( \frac{\sinh a_i'x}{\sinh a_ix}\frac{\sinh b_i x}{\sinh b_i' x}\bigr) ,  \]
we can get the positive definiteness of $\prod_{i=1}^n\frac{\sinh a_i'x}{\sinh b_i'x}$.
\end{proof}

\vspace{5mm}


\begin{lemma}
Let $p$ be a prime number with $p> 3$.
For any positive integer $n$, we have
\[  f(x) = \frac{\sinh ((p+1)x/p) }{\sinh x} \frac{(\sinh x/p)^n}{(\sinh ((p-1)x/p))^n}   \]
is not positive definite.
\end{lemma}
\begin{proof}
We prove the Fourier transform of $f$ is not positive by using the similar method 
in \cite{hkpr}:Lemma 5.2.
Set the function
\[  f(z) = \frac{\sinh ((p+1)z/p)}{ \sinh z} \frac{(\sinh z/p)^n}{(\sinh ((p-1)z/p))^n}  \]
and define the closed curve $C_1+C_2+C_3+C_4$ in $\mathbb{C}$ as follows:
\begin{align*}
   C_1: & \; z = x, \; x:-R\to R,  &  C_2: & \; z=R+iy, \; y:0\to p\pi, \\
   C_3: & \; z = x+p \pi i, \; x:R \to -R, & C_4: & \; z=-R+iy, \; y:p\pi \to 0.
\end{align*}
for $R>0$.

We have
\[  \int_{C_1+C_3} e^{izs} f(z) dz = (1+(-1)^ne^{-p\pi s}) \int_{-R}^R e^{ixs}f(x)dx  ,  \]
by the relation
\begin{align*}
  \sinh\frac{p+1}{p}(x+p\pi i) &= \sinh (\frac{p+1}{p}x), &
  \sinh\frac{1}{p}(x+p\pi i) &= -\sinh (\frac{1}{p}x),  \\
  \sinh (x+p\pi i) &= - \sinh x , &
  \sinh\frac{p-1}{p}(x+p\pi i) &= \sinh (\frac{p-1}{p}x).
\end{align*}
When $|\Re z|\ge \log \sqrt{2}$, 
\[   \frac{e^{|\Re z|}}{4} \le |\sinh z| \le e^{|\Re z|}  .   \]
Using this relation, we have, for a sufficiently large $R$,
\[  |e^{i(\pm R+iy)s} f(\pm R+iy)| \le e^{-ys}\times 
         \frac{ 4^{n+1} e^{(p+1)R/p} e^{nR/p} }{e^R e^{n(p-1)R/p } }
         =4^{n+1}e^{-ys}e^{(\frac{2n+1}{p}-n)R}. \]
So we can get
\[  \lim_{R\to \infty} \int_{C_1+C_2+C_3+C_4} e^{izs}f(z)dz =
     (1+(-1)^ne^{-p\pi s}) \hat{f}(s).  \]

The singular points of $f(z)$ in the rectangle $C_1+C_2+C_3+C_4$ are contained in
\[  \{k\pi i : k=0,1,2,\ldots, p\} \cup \{\frac{lp}{p-1}\pi i : l=0,1,2,\ldots,p-1\}. \]
We can see that $0$ and $p\pi i$ are removable singularities, each $k\pi i$ ($k\in \{1,2,\ldots, p-1\}$)
is a pole of order $1$,  each $\frac{lp}{p-1}\pi i$ ($l\in (\{1,2,\ldots,p-2\} \setminus \{\frac{p-1}{2} \} ) $) is a pole of order $n$,
and $\frac{p}{2}\pi i$ is a pole of order $n-1$.
For real numbers $\alpha, \beta$, we have
\[  e^{izs}  = e^{is(z-\alpha i)}e^{-\alpha s} = \sum_{k=0}^\infty c(\alpha i, k)(z-\alpha i)^k  \]
and
\begin{align*}
   \sinh \beta z & = \sinh (\beta (z-\alpha i) + \beta\alpha i)  \\
         &= \cos (\beta\alpha)  \sinh(\beta(z-\alpha i)) +i\sin (\beta\alpha) \cosh (\beta(z-\alpha i))  \\
         &= \sum_{k=0}^\infty d_k(\beta, \alpha i) (z-\alpha i)^k,
\end{align*}
where
\[  c(\alpha i, k) =\frac{(is)^k}{k!}e^{-\alpha s}, \quad
    d_k(\beta, \alpha i) = \begin{cases} \frac{\beta^k}{k!}i\sin (\beta\alpha)  &  (k: \text{ even}) \\
                                                             \frac{\beta^k}{k!}\cos (\beta\alpha)  &  (k: \text{ odd})  \end{cases} . \]
When $\alpha \in \{\pi, 2\pi, \ldots, (p-1)\pi \}$, the residue ${\rm Res}(e^{izs} f(z): \alpha i)$ of $e^{izs}f(z)$ at $\alpha i$ is 
\begin{align*}
   {\rm Res}(e^{izs}f(z): \alpha i)  
  & = \frac{c(\alpha i,0)d_0(\frac{p+1}{p}, \alpha i) d_0(\frac{1}{p}, \alpha i)^n}{d_1(1, \alpha i) d_0(\frac{p-1}{p}, \alpha i)^n}  \\
  & = e^{-\alpha s} \frac{i (\sin \frac{\alpha}{p})^n \sin(\frac{p+1}{p}\alpha) } { \cos \alpha (\sin (\frac{p-1}{p}\alpha) )^n} .
\end{align*}
When $\alpha \in ( \{\frac{p\pi}{p-1}, \frac{2p\pi}{p-1}, \ldots, \frac{(p-2)p\pi}{p-1} \}\setminus 
\{\frac{p}{2}\pi \} )$,
$f(z)(z-\alpha i)^n$ is analytic at $\alpha i$.
So this has the Taylor expansion at $\alpha i$ as follows:
\[  f(z)(z-\alpha i)^n = \sum_{k=0}^\infty e_k(\alpha)(z -\alpha i)^k ,  \]
where we remark that $e_k(\alpha)$ does not depend on $s$.
So we can compute
\[  {\rm Res}(e^{izs}f(z): \alpha i)  = \sum_{k=0}^{n-1} c(\alpha i, k)e_{n-1-k}(\alpha)
   = e^{-\alpha s} (\sum_{k=0}^{n-1}\frac{e_{n-1-k}(\alpha)}{k!}(is)^k ).  \]
Because $\frac{p}{2}\pi i$ is a pole of order $n-1$, by the similar argument, we have
\[   {\rm Res}(e^{izs}f(z): \frac{p}{2}\pi i)  = \sum_{k=0}^{n-2} c(\frac{p}{2}\pi i, k)e_{n-2-k}
   = e^{-p\pi s/2} (\sum_{k=0}^{n-2}\frac{e_{n-2-k}}{k!}(is)^k )  \]
for a suitable numbers $\{e_n\}$. 
By the Cauchy Residue Theorem, we have
\[ (1+(-1)^ne^{-p\pi s})\hat{f}(s) =  2\pi i \bigl( \sum_{k=1}^{p-1} {\rm Res}(e^{izs}f(z):k\pi i) + 
   \sum_{l=1}^{p-2} {\rm Res}(e^{izs} f(z): \frac{lp\pi i}{p-1}) \bigr) .\]
Then we have
\begin{align*}
   & e^{\pi s} (1+(-1)^ne^{-p\pi s})\hat{f}(s) \\
 = & 2\pi \frac{ (\sin \frac{\pi}{p})^n \sin(\frac{p+1}{p}\pi)}{(\sin (\frac{p-1}{p}\pi))^n} 
      -2\pi \sum_{k=2}^{p-1}e^{(1-k)\pi s}\frac{(\sin\frac{k\pi}{p})^n \sin\frac{k(p+1)\pi}{p} }{(\cos k\pi) (\sin \frac{k(p-1)\pi}{p})^n} \\
   & \qquad + 2 \pi i \sum_{l=1, l\neq (p-1)/2}^{p-2} e^{(1-\frac{lp}{p-1})\pi s} 
     (\sum_{k=0}^{n-1}\frac{e_{n-1-k}(\frac{lp\pi}{p-1})}{k!} (is)^k)  \\
   & \qquad +2 \pi i   e^{(1-p/2)\pi s} (\sum_{k=0}^{n-2}\frac{e_{n-2-k}}{k!}(is)^k ).
\end{align*}
When $s$ tends to $\infty$, then the right-hand side of above identity tends to
\[   2\pi \frac{ (\sin \frac{\pi}{p})^n \sin(\frac{p+1}{p}\pi)}{(\sin (\frac{p-1}{p}\pi))^n}<0,  \]
that is,  $\hat{f}(s)$ is not positive for a sufficiently large $s$.
This means that $f$ is not positive definite.
\end{proof}

\vspace{5mm}


\begin{proposition}
Let $a_i, b_i>0$  $(i=1,2,\ldots,n)$ with $a_1+a_2+\cdots + a_n>b_1+b_2+\cdots +b_n$.
Then we have
\[  \prod_{i=1}^n \frac{\sinh a_i x}{\sinh b_i x}  \]
is not positive definite.
\end{proposition}
\begin{proof}
We set $f(x) = \prod_{i=1}^n \frac{b_i \sinh a_i x}{a_i\sinh b_i x}$. 
It suffices to show $f$ is not positive definite.
Clearly we have $f(0)=1$, and $f(x)=f(-x)$. 
Since $a_1+a_2+\cdots +a_n >b_1+b_2+\cdots +b_n$, it follows $\lim_{x\to \infty} f(x)=\infty$.
Then the self-adjoint matrix 
\[  \begin{pmatrix} f(0) & f(-x)  \\ f(x) & f(0) \end{pmatrix}  \]
 is not positive for a sufficiently large $x$.
So $f$ is not positive definite.
\end{proof}

\vspace{5mm}


\begin{proposition}
Let $a_i, b_i>0$  $(i=1,2,\ldots,n)$ with $a_1> \max \{b_1, b_2, \ldots , b_n\}$.
Then we have
\[  \prod_{i=1}^n \frac{\sinh a_i x}{\sinh b_i x}  \]
is not positive definite.
\end{proposition}
\begin{proof}
We may assume that $a_1\ge a_2 \ge \ldots \ge a_n$, $b_1=1\ge b_2 \ge \ldots \ge b_n$, and $a_1>1$.
We can choose a prime number $p$ such that 
\[  p> 3, \quad a_1>\frac{p+1}{p-1}, \text{ and } a_n> \frac{1}{p-1} .  \]
By Lemma 2.5 and Lemma 2.4(2), we have
\[  \frac{ (\sinh \frac{p+1}{p}x) (\sinh \frac{x}{p})^{n-1} }{ (\sinh \frac{p-1}{p}x )^n}  \]
is not positive definite, because  so is
\[  \frac{ (\sinh \frac{p+1}{p}x) (\sinh \frac{x}{p})^{n-1} }{ (\sinh x) (\sinh \frac{p-1}{p} x)^{n-1} }  \]
and $(p-1)/p<1$.
Substituting $\frac{p}{p-1}x$ for $x$, we have
\[   \frac{ (\sinh \frac{p+1}{p-1}x) (\sinh \frac{x}{p-1})^{n-1} }{ (\sinh x)^n }  \]
is not positive definite.
Using Lemma 2.4(2), 
\[   \prod_{i=1}^n \frac{\sinh a_i x}{\sinh b_i x}   \]
is not positive definite, since $a_1>\frac{p+1}{p-1}$, $a_2\ge \ldots \ge a_n >\frac{1}{p-1}$ and $1=b_1\ge b_2 \ge \ldots  \ge b_n$.
\end{proof}

\vspace{5mm}

For $n$-tuples of positive numbers $a=(a_1,a_2,\ldots,a_n)$ and $b=(b_1,b_2,\ldots,b_n)$,
we call that $a$ is weakly submajorised by $b$ ($a \preceq_w b$) if there exists permutations $\sigma, \tau$ on $\{ 1,2,\ldots,n \}$ satisfying
\begin{gather*}
  a_{\sigma(1)}\ge  a_{\sigma(2)}\ge \ldots \ge a_{\sigma(n)}, \;
  b_{\tau(1)}\ge b_{\tau(2)} \ge \ldots \ge b_{\tau(n)}  \\
  \text{and } \sum_{i=1}^k a_{\sigma(i)} \le \sum_{i=1}^k b_{\tau(i)} \; 
  \text{ for any } k \in \{1,2,\ldots, n\}.
\end{gather*} 

\vspace{5mm}


\begin{theorem}
Let $a_i, b_i>0$  $(i=1,2,\ldots,n)$.
If $(a_1,a_2,\ldots , a_n)\preceq_w (b_1,b_2, \ldots, b_n)$, then
\[  \prod_{i=1}^n \frac{\sinh a_i x}{\sinh b_i x}   \]
is infinitely divisible.
\end{theorem}
\begin{proof}
It suffices to show that the function
\[   \prod_{i=1}^n \frac{\sinh a_i x}{\sinh b_i x}  \]
is infinitely divisible if $a_1\ge a_2 \ge \ldots \ge a_n$, $b_1\ge b_2 \ge \ldots \ge b_n$,
and $\sum_{i=1}^k a_i \le \sum_{i=1}^k b_i$ for all $k \in \{1,2,\ldots, n\}$.
We prove this statement using induction on $n$.

When $n=1$, it follows from Lemma 2.2 since $a_1\le b_1$.

We assume the statement is valid for some $n$
and $a_1\ge a_2 \ge \ldots \ge a_n \ge a_{n+1}$, $b_1\ge b_2 \ge \ldots \ge b_n\ge b_{n+1}$,
and $\sum_{i=1}^k a_i \le \sum_{i=1}^k b_i$ for all $k \in \{1,2,\ldots, n+1\}$.
If $a_i \le b_i$ for all $i=1,2,\ldots, n+1$, then 
\[    \prod_{i=1}^{n+1} \frac{\sinh a_i x}{\sinh b_i x}  \]
is infinitely divisible by Lemma 2.1(3) and Lemma 2.2.
If $a_j>b_j$ for some $j$, then we may assume that
\[  a_k \le b_k \; (k=1,2,\ldots,j-1) \text{ and } a_j>b_j.  \]
We have
\begin{align*}
    \prod_{i=1}^{n+1} \frac{\sinh a_i x}{\sinh b_i x} =
     \prod_{i=1}^{j-2} \frac{\sinh a_i x}{\sinh b_i x} & \times \frac{\sinh a_{j-1} x}{\sinh (b_{j-1}+b_j -a_j)x}
     \times \prod_{i=j+1}^{n+1} \frac{\sinh a_i x}{\sinh b_i x}  \\
    & \times \frac{\sinh a_j x \sinh (b_{j-1}+b_j-a_j)x}{\sinh b_{j-1}x \sinh b_j x} .
\end{align*}
By the assumption of induction, we can see
\[  \prod_{i=1}^{j-2} \frac{\sinh a_i x}{\sinh b_i x} \times \frac{\sinh a_{j-1} x}{\sinh (b_{j-1}+b_j -a_j)x}
     \times \prod_{i=j+1}^{n+1} \frac{\sinh a_i x}{\sinh b_i x}  \]
is infinitely divisible and by Lemma 2.3
\[  \frac{\sinh a_j x \sinh (b_{j-1}+b_j-a_j)x}{\sinh b_{j-1}x \sinh b_j x}  \]
is also infinitely divisible since $b_{j-1} \ge \max\{a_j, b_{j-1}+b_j-a_j\}$ and $b_{j-1}+b_j = a_j +(b_{j-1}+b_j-a_j)$.
By Lemma 2.1 we can prove
\[  \prod_{i=1}^{n+1} \frac{\sinh a_i x}{\sinh b_i x}  \]
is infinitely divisible.
\end{proof}

\vspace{5mm}


\begin{example}
Let $a_1=8$, $a_2=6$, $a_3=3$, $b_1=9$, $b_2=4$, $b_3=4$.
It does not satisfies the assumption of Proposition 2.6 and Theorem 2.8.
We can show the function
\[   f(x) = \frac{\sinh 8x \sinh 6x \sinh 3x}{\sinh 9x \sinh 4x \sinh 4x}  \]
is not positive definite.
\end{example}

We remark that $f(0)=1$ and $f(x)=f(-x)$.
We can get the following approximation values:
\begin{gather*}
    |f(1/3) -0.9780192940|\le 10^{-10}, \; |f(2/3)-0.9908829679|\le 10^{-10}, \\
    \text{and }    |f(1) - 0.9981846167|\le 10^{-10} .  
\end{gather*}
Since $|f(0)|, |f(1/3)|, |f(2/3)|, |f(1)|\le 1$, we can get the following estimation:
\[   | \det \begin{pmatrix} f(0) & f(1/3) & f(2/3) & f(1) \\ f(1/3) & f(0) & f(1/3) & f(2/3) \\
                                         f(2/3) & f(1/3) & f(0) & f(1/3) \\ f(1) & f(2/3) & f(1/3) & f(0) \end{pmatrix} - (-0.0000095)| \le 10^{-7}  \]
by using these approximation values.
This means that the $4\times 4$ matrix $( f(\frac{i-j}{3} ) )_{i,j=1}^4$ is not positive. 
So we have $f$ is not positive definite.

\vspace{5mm}


\begin{example}
Let $a_1=8$, $a_2=6$, $a_3=1$, $b_1=9$, $b_2=4$, $b_3=4$.
It also does not satisfies the assumption of Proposition 2.6 and Theorem 2.8.
But we can show the function
\[   \frac{\sinh 8x \sinh 6x \sinh x}{\sinh 9x \sinh 4x \sinh 4x}  \]
is infinitely divisible.
\end{example}

We also use the following integral expression:
\[  \log \frac{b\sinh ax}{a\sinh bx} = \int_{-\infty}^\infty (e^{ixt}-1-ixt)
           \frac{\sinh((1/a-1/b)\pi t/2)}{2t\sinh(\pi t/2a)\sinh(\pi t/2b)}dt  \]
by Kosaki(\cite{kosaki}: Corollary 3).
Since
\[  \log \frac{3\sinh 8x \sinh 6x \sinh x}{\sinh 9x \sinh 4x \sinh 4x}
   = \log \frac{9\sinh 8x}{8\sinh 9x} -\log \frac{6\sinh 4x}{4\sinh 6x} + \log \frac{4\sinh x}{\sinh 4x}, \]
we have
\[  \log \frac{3\sinh 8x \sinh 6x \sinh x}{\sinh 9x \sinh 4x \sinh 4x}
    = \int_{-\infty}^\infty \frac{e^{ixt}-1-ixt}{t^2} F(t)dt, \]
where $F(t) = f_1(t)-f_2(t)+f_3(t)$ and $f_i$'s are non-negative integrable functions as follows:
\begin{gather*} f_1(t) = \frac{t^2\sinh (\pi t/144)}{2t \sinh (\pi t/16) \sinh (\pi t/18) } , \quad
    f_2(t) = \frac{t^2\sinh (\pi t/24)}{2t \sinh (\pi t/12) \sinh (\pi t/8) }, \\
    f_3(t) = \frac{t^2\sinh (3\pi t/8)}{2t \sinh (\pi t/2) \sinh (\pi t/8) }.  
\end{gather*}
We set
\[  F(t) = \frac{t^2g(t)}{2t\sinh (\pi t/18)\sinh(\pi t/16)\sinh(\pi t/12)\sinh(\pi t/8)\sinh(\pi t/2)} , \]
where
\begin{align*}
  g(t) =& \sinh(\frac{\pi t}{144})\sinh(\frac{\pi t}{12})\sinh(\frac{\pi t}{8})\sinh(\frac{\pi t}{2}) \\
   &  -\sinh(\frac{\pi t}{24})\sinh(\frac{\pi t}{16})\sinh(\frac{\pi t}{18})\sinh(\frac{\pi t}{2}) \\
   &  +\sinh(\frac{3\pi t}{8})\sinh(\frac{\pi t}{16})\sinh(\frac{\pi t}{18})\sinh(\frac{\pi t}{12}).
\end{align*}
If we show that $g(t)\ge 0$ for all $t \in \mathbb{R}$, then so is $F(t)$. 
This implies the infinite divisibility of the function
\[  \frac{\sinh 8x \sinh 6x \sinh x}{\sinh 9x \sinh 4x \sinh 4x}.  \]
           
By the formulas 
\begin{gather*}  \sinh a \sinh b = \frac{1}{2} (\cosh (a+b)- \cosh (a-b)), \\
     \cosh a \cosh b = \frac{1}{2} (\cosh (a+b) + \cosh (a-b)), 
\end{gather*} 
we have
\begin{align*}
     \sinh a & \sinh b \sinh c \sinh d \\
  = \frac{1}{8} & ( \cosh (a+b+c+d) + \cosh (a+b-c-d)  \\
       & - \cosh (a-b+c+d) - \cosh (a-b-c-d) \\
       & - \cosh (a+b+c-d) - \cosh (a+b-c+d) \\
       & + \cosh (a-b+c-d) +\cosh (a-b-c+d) ). 
\end{align*}
Using this relation, we can get
\begin{align*}
  g(144s/\pi) =& (\sinh(s)\sinh(12s)\sinh(18s)\sinh(72s) \\
   & \qquad -\sinh(6s)\sinh(9s)\sinh(8s)\sinh(72s) \\
   & \qquad +\sinh(54s)\sinh(9s)\sinh(8s)\sinh(12s)  ) \\
   =& \frac{1}{8} (\cosh 103s + 2\cosh 83s +2\cosh 77s +2\cosh 49s +2\cos 43s \\
   & \qquad -\cosh 101s-\cosh 95s -3\cosh 67s -\cosh 65s \\
   & \qquad -\cosh 61s -\cosh 59s -\cosh 25s ) \\
   =& \frac{1}{8} \sum_{k=0}^\infty \frac{c_k}{(2k)!}s^{2k} ,
\end{align*}
where
\begin{align*}
  c_k =&  103^{2k} +2\cdot 83^{2k}+2\cdot 77^{2k} + 2\cdot 49^{2k} + 2\cdot 43^{2k} \\
   & -101^{2k}-95^{2k}-3\cdot 67^{2k}-65^{2k}-61^{2k}-59^{2k} -25^{2k}.
\end{align*}
Since
\begin{align*}
  \frac{c_9}{103^{18}} & \ge 1-(\frac{101^{18}+95^{18}+3\cdot 67^{18}
     +65^{18}+61^{18}+59^{18}+25^{18}}{103^{18}} ) (\approx 0.062)  \\
    &\ge 0,
\end{align*}
we have $c_k >0$ if $k\ge 9$.
By the direct computation, we can get
\[  c_0, c_1, c_2, \ldots , c_8 \ge 0.  \]
So $g(t)$ is non-negative for all $t\in\mathbb{R}$.

\vspace{5mm}


\noindent
{\bf Remark.}  We have already shown in Lemma 2.3 that the function
\[  f(x) = \frac{bd \sinh ax \sinh cx}{ac\sinh bx \sinh dx}  \]
is infinitely divisible for any positive numbers $a,b,c,d$ with $b>\max\{a,c\}$ and
$a+c=b+d$.
As stated in \cite{kosaki}:Theorem 5,   the density function $F$ appeared 
in the integral expression as below 
becomes even, positive and integrable (i.e., $F(t)$  admits a finite limit at the origin and rapidly decreasing at $\infty$) : 
\begin{gather*}
    \log f(x)  = \int_{-\infty}^\infty (e^{ixt}-1-ixt)F(t) dt    \\
\intertext{ and}
    F(t)  = \frac{\sinh ((b-a)\pi t/(2ab)) }{ 2t\sinh (\pi t/2a) \sinh(\pi t/2b) }
       - \frac{\sinh ((c-d) \pi t/ (2cd)) } {2t\sinh (\pi t/2c) \sinh(\pi t/2d)}  .
\end{gather*}
 
When $n\ge 2$ and two sequences $\alpha =(a_1,a_2, \ldots,a_n)$ and $\beta = (b_1,b_2, \ldots, b_n)$ 
of positive numbers satisfy the following condition:
\[  \sum_{i=1}^k a_i \le \sum_{i=1}^k b_i \; (k=1,2,\ldots,n-1) \text{ and }
     \sum_{i=1}^n a_i = \sum_{i=1}^n b_i,  \]
the function
\[  g(x) = \prod_{i=1}^n \frac{b_i\sinh a_i x}{a_i \sinh b_i x}  \]
is also infinitely divisible by Theorem 2.8.
By the argument in the proof of Theorem 2.8, we can see that $g(x)$ is written by the product  
$f_1(x)f_2(x)\cdots f_{n-1}(x)$ of $f_1(x)$, $f_2(x)$, $\ldots$, and $f_{n-1}(x)$, where each $f_i(x)$ has the form 
\[  f(x) = \frac{bd \sinh ax \sinh cx}{ac\sinh bx \sinh dx} \quad
    ( b>\max\{a,c\} \text{ and } a+c=b+d). \]
For examples, we have the following expressions:
\begin{gather*}
    \frac{\sinh 6x \sinh 5x \sinh 3x}{\sinh 9x \sinh 4x \sinh x} =
    \frac{\sinh 6x \sinh 3x}{\sinh 8x \sinh x}\times \frac{\sinh 5x \sinh 8x}{\sinh 9x \sinh 4x}, \\
    \frac{\sinh 7x \sinh 5x \sinh 4x}{\sinh 9x \sinh 6x \sinh x} = 
    \frac{\sinh 7x \sinh 5x}{\sinh 9x \sinh 3x} \times \frac{\sinh 4x \sinh 3x}{\sinh 6x \sin x}. 
\end{gather*}
This means that, the density function $G$ appeared in the integral expression as below  is 
also even, positive and integrable:
\[  \log g(x) = \int_{-\infty}^\infty (e^{ixt}-1-ixt)G(t) dt.  \]
Since $\int_{-\infty}^\infty tG(t) dt =0$,  the function $g(x)^r$ has the following form for any $r>0$:
\[  g(x)^r = e^{r\varphi(x)}C^r  ,  \]
where $\varphi (x) =\hat{G}(x)= \int_{-\infty}^\infty e^{ixt}G(t)dt$ and $C=\exp(-\int_{-\infty}^\infty G(t)dt)$.
We can also see that $g(x)$ is infinitely divisible by Lemma 2.1(4).


\section{Proof of Theorems and Applications}

For $a,b\in \mathbb{R}$, we define
\[   f_{a,b}(t) = t^{\gamma(a,b)}\frac{b(t^a-1)}{a(t^b-1)} , \quad \gamma(a,b)=\frac{1 -a+b}{2}, \]
where we use the notation $(t^a-1)/a = \log t$ if $a=0$.
Then the function $f_{a,b}:(0,\infty)\longrightarrow (0,\infty)$ is continuous with $f_{a,b}(1)=1$ 
(i.e., $f_{a,b}\in C(0,\infty)^+_1$) and symmetric ($f_{a,b}(t) = tf_{a,b}(1/t)$).
It is clear that
\[  f_{a,b}(t) = f_{-a,b}(t) = f_{a,-b}(t) =f_{-a,-b}(t) \text{ and } f_{a,a}(t)=\sqrt{t}.  \]
So we only consider the case $a,b\ge0$.

For $\alpha=(a_1,a_2,\ldots,a_n)$, $\beta =(b_1,b_2,\ldots, b_n) \in \mathbb{R}^n$,
we define the function as follows:
\[    f_{\alpha, \beta}(t) = t^{\gamma(\alpha, \beta)}\prod_{i=1}^n \frac{b_i(t^{a_i}-1)}{a_i(t^{b_i}-1)}, \]
where $\gamma= (1-\sum_{i=1}^n (a_i-b_i) )/2$ and we also use the notation
$(t^a-1)/a=\log t$ if $a=0$.
Then the function $f_{\alpha,\beta}$ also satisfies $f_{\alpha,\beta}\in C(0,\infty)^+_1$ and $f_{\alpha,\beta}(t)=tf_{\alpha,\beta}(1/t)$.
If we define $\tilde{\alpha} =(-a_1,a_2,\ldots,a_n)$, that is, $\tilde{\alpha}$ is replaced $a_1$ by $-a_1$ in $\alpha$,
then we have
\begin{align*}
  f_{\tilde{\alpha},\beta}(t) & = t^{\gamma(\tilde{\alpha}, \beta)} \frac{b_1(t^{-a_1}-1)}{(-a_1)(t^{b_1}-1)}
                                 \prod_{i=2}^n \frac{b_i(t^{a_i}-1)}{a_i(t^{b_i}-1)}  \\
  & = t^{\gamma(\tilde{\alpha}, \beta)} t^{-a_1} \frac{b_1(t^{a_1}-1)}{a_1(t^{b_1}-1)}
                                 \prod_{i=2}^n \frac{b_i(t^{a_i}-1)}{a_i(t^{b_i}-1)}  \\
  & = t^{\gamma(\alpha, \beta)}\prod_{i=1}^n \frac{b_i(t^{a_i}-1)}{a_i(t^{b_i}-1)} =  f_{\alpha,\beta}(t).
\end{align*}
This means $f_{\alpha, \beta}(t) = f_{|\alpha|, |\beta|}(t)$, where 
 $|\alpha| = (|a_1|, |a_2|, \ldots, |a_n|)$.

\vspace{5mm}


\noindent
{\bf Proof of Theorem 1.1} \; For $\alpha,\beta\in \mathbb{R}^n$ and 
$\alpha', \beta' \in \mathbb{R}^m$, it suffices to show that the function
\[  h: \mathbb{R} \ni x \mapsto \frac{f_{\alpha,\beta}(e^{2x})}{f_{\alpha',\beta'}(e^{2x})} \in (0,\infty) \]
is positive definite.
By Lemma 2.1(2) and the fact $f_{\alpha,\beta} = f_{|\alpha|, |\beta|}$, we may assume that
each component of $\alpha,\beta,\alpha'$ and $\beta'$ is positive.

By the calculation
\begin{align*}
   f_{\alpha,\beta}(e^{2x}) & = e^{2\gamma(\alpha,\beta)x}\prod_{i=1}^n 
         \frac{b_i(e^{2a_ix}-1)}{a_i(e^{2b_ix}-1)}  \\
      & = e^{2\gamma(\alpha,\beta)x} e^{(\sum_{i=1}^n (a_i-b_i))x} \prod_{i=1}^n
         \frac{ b_i(e^{a_ix}-e^{-a_ix}) }{ a_i(e^{b_ix}-e^{-b_ix}) }  \\
      & = e^x \prod_{i=1}^n \frac{b_i\sinh a_ix}{a_i\sinh b_ix},
\end{align*}
the function $h(x)$ has the following form:
\[  h(x) = \prod_{i=1}^n \frac{b_i\sinh a_ix}{a_i\sinh b_ix}
             \prod_{j=1}^m \frac{c_j\sinh d_jx}{d_j\sinh c_jx}  .  \]
By Theorem 2.8, $h(x)$ is infinitely divisible, in particular positive definite if $(a_1,a_2,\ldots,a_n,d_1,d_2,\ldots,d_m) \preceq_w
(b_1,b_2,,\ldots, b_n,c_1,c_2,\ldots,c_m)$. \hfill $\Box$

\vspace{5mm}

For $\alpha, \beta\in \mathbb{R}^n$, we set $M_{\alpha, \beta}(s,t)=t f_{\alpha,\beta}(s/t)$.
Then $M_{\alpha, \beta}(s,t)$ can be written as follows:
\begin{align*}
    M_{\alpha, \beta}(s,t) & =(st)^{\gamma(\alpha, \beta)} 
               \prod_{i=1}^n \frac{b_i(s^{a_i}-t^{a_i})}{a_i(s^{b_i}-t^{b_i})}  \\ 
      & = (st)^{1/2} \prod_{i=1}^n 
              \frac{b_i \sinh (a_i(\log s-\log t)/2)}{ a_i \sinh (b_i (\log s-\log t)/2)}. 
\end{align*}
Let $k$ be a positive integer smaller than $n$.
For $1\le i_1 <i_2<\cdots <i_k\le n$, we define 
$\alpha\setminus (i_1, i_2,\ldots,i_k) \in \mathbb{R}^{n-k}$ by deleting the $i_1$-th, $i_2$-th, \ldots
and $i_k$-th components from $\alpha$.
If $\alpha, \beta\in \mathbb{R}^n$ satisfy the relation $|\alpha|\preceq_w |\beta|$, then
we have
\[ \bigl( f_{\alpha\setminus(i_1,i_2,\ldots,i_k), \beta\setminus(j_1,j_2,\ldots,j_k)}(t)\bigr)^r
    \preceq \bigl( f_{(b_{j_1}, b_{j_2},\ldots,b_{j_k}), (a_{i_1}, a_{i_2}, \ldots, a_{i_k})}(t)\bigr)^r  \]
for any $r>0$ by Theorem 1.1, where it is also assumed $1\le j_1<j_2<\cdots <j_k\le n$.
In the case $r=1$,  we can get the following operator norm inequality:
\begin{align*}
     ||| M_{\alpha\setminus(i_1,i_2,\ldots,i_k), \beta\setminus(j_1,j_2,\ldots,j_k)}&(L_H,R_K)X|||  \\
    \le &  ||| M_{(b_{j_1}, b_{j_2},\ldots,b_{j_k}), (a_{i_1}, a_{i_2}, \ldots, a_{i_k})}(L_H,R_K)X |||  
\end{align*}
for any $H,K,X\in \mathbb{M}_N(\mathbb{C})$ with $H,K>0$.

As an example, we consider $\alpha = (8,8,7,5,3)$ and $\beta = (10,9,6,4,2) \in \mathbb{R}^5$.
It is clear $\alpha \preceq_w \beta$.
If we choose as $i_1=1, i_2=4, j_1=2$ and $j_2=5$, then we  have
\[   ||| M_{(8,7,3),(10,6,4)}(L_H,R_K)X ||| \le
      ||| M_{(9,2),(8,5)}(L_H,R_K)X|||.  \]  
By using our method,  if we choose $i_1=2$ and $j_1=2$ for $\alpha =(1,1)\preceq_w (1,2)=\beta$,
then we can get McIntosh's inequality
\[  ||| H^{1/2}XK^{1/2} ||| \le \frac{1}{2} |||HX+XK ||| \]
for all $H,K,X\in \mathbb{M}_N(\mathbb{C})$ with $H, K>0$,
because $t^{1/2} = f_{(1),(1)}(t) \preceq f_{(2),(1)}(t) = (1+t)/2$.

\vspace{5mm}


\noindent
{\bf Proof of Theorem 1.2} \; At first, we show that, for $a_1\ge a_2 > 0$ and $b_1\ge b_2>0$,
\[  f_{a_1,b_1} \preceq f_{b_2,a_2} \Leftrightarrow \; a_1 \le b_1 \text{ and } a_1+a_2 \le b_1+b_2 . \]
We have already seen that $f_{a_1,b_1} \preceq f_{b_2,a_2}$ is equivalent to 
\[  \frac{\sinh a_1x \sinh a_2x}{\sinh b_1x \sinh b_2 x}  \]
is positive definite.
The implication $(\Leftarrow)$ follows from Theorem 2.8.
If $a_1>b_1$ or $a_1+a_2>b_1+b_2$, $\frac{\sinh a_1x \sinh a_2x}{\sinh b_1x \sinh b_2 x} $ is not
positive definite by Proposition 2.7 and Proposition 2.6.
So the reverse implication $(\Rightarrow)$ is valid.

(1) When $a\ge b$, we have
\begin{align*}
  f_{a,b}\preceq f_{c,d} & \Leftrightarrow 
        \frac{\sinh ax \sinh dx}{\sinh bx \sinh cx} \text{ is positive definite}. \\
        & \Leftrightarrow  a\le c \text{ and } a+d \le b+c .  \\
        & \Leftrightarrow  (c,d) \in \{(x,y) : x \ge a, 0 \le y \le x-a+b \}.
\end{align*}

(2) When $a < b$, we have
\begin{align*}
  f_{a,b}\preceq f_{c,d} & \Leftrightarrow 
        \frac{\sinh ax \sinh dx}{\sinh bx \sinh cx} \text{ is positive definite}. \\
        & \Leftrightarrow  (d \le c) \text{ or } (d\le b \text{ and }a+d \le b+c).  \\
        & \Leftrightarrow  (c,d) \in \{(x,y): 0\le y\le x\}  \\
        & \qquad \qquad \qquad \cup \{(x,y) : 0\le x \le y \le x-a+b, y\le b\}. \qquad \Box
\end{align*}

\vspace{1cm}
\begin{center}
\begin{tabular}{l l}
Imam Nugraha Albania  &   [ Current Address ]  \\
Department of Mathematics Education  &  Graduate School of Science   \\
Universitas Pendidikan Indonesia (UPI)  &  Chiba University \\
Jl. Dr Setia Budhi 229 Bandung 40154    &  Chiba 263-8522\\
Indonesia   & Japan \\
e-mail address: albania@upi.edu   & e-mail address: phantasion@gmail.com  \\
  &  \\
  &  \\
Masaru NAGISA &  \\
Graduate School of Science & \\
Chiba University & \\
Chiba 263-8522 & \\
Japan & \\
e-mail address: nagisa@math.s.chiba-u.ac.jp & 
\end{tabular}
\end{center}

\vspace{5mm}

{\it MSC} primary: 47A53, 47A64,  secondary: 15A42, 15A45, 15A60

\vspace{5mm}

{\it Keywords:}  McIntosh's inequality; Operator norm inequality;Unitarily invariant norm;Positive definite function; Infinitely divisible function

\end{document}